\documentclass[a4paper]{siamart190516} 
\usepackage{damacros} 
\usepackage{enumitem}

\newsiamthm{example}{Example} \newsiamthm{problem}{Problem}
\crefname{problem}{Problem}{Problems} \Crefname{problem}{Problem}{Problems}
\renewcommand{\phi}{\varphi} \newcommand\epsi\varepsilon
\newcommand\wto\rightharpoonup

\renewcommand\epsilon\varepsilon

\newcommand{\TheTitle}{Well-posedness and regularity of solutions to neural field
  problems with dendritic processing.} \newcommand{\TheShortTitle}{} \newcommand{\TheAuthors}{}
\headers{\TheShortTitle}{\TheAuthors}
\newcommand{\changed}[1]{#1}

\title{{\TheTitle}}

\author{Daniele Avitabile,
  \thanks{Amsterdam Centre for Dynamics and Computation, Department of Mathematics,
  Vrije Universiteit Amsterdam, The Netherlands. Inria MathNeuro, Montpellier,
  France.
\email{d.avitabile@vu.nl}} 
  \and 
  Nikolai~V. Chemetov,
  \thanks{Department of Computing and Mathematics, University of S{\~a}o
Paulo, 14040-901 Ribeir{\~a}o Preto - SP, Brazil, \email{nvchemetov@gmail.com}} 
\and 
P.~M. Lima,\thanks{Instituto Superior Tecnico, University of Lisbon
\email{pedro.t.lima@ist.utl.pt}}} 
\graphicspath{{Figures/}}

\begin{document} 

\maketitle

\begin{abstract}
  We study solutions to a recently proposed neural field model in which dendrites are
  modelled as a continuum of vertical fibres stemming from a somatic layer. Since
  voltage propagates along the dendritic direction via a cable equation with nonlocal
  sources, the model features an anisotropic diffusion operator, as well as an
  integral term for synaptic coupling. The corresponding Cauchy problem is thus
  markedly different from classical neural field equations. We prove that the weak
  formulation of the problem admits a unique solution, with embedding estimates
  similar to the ones of nonlinear local reaction-diffusion equations. Our analysis
  relies on perturbing weak solutions to the diffusion-less problem, that is, a
  standard neural field, for which weak problems have not been studied to date. We
  find rigorous asymptotic estimates for the problem with and without diffusion, and
  prove that the solutions of the two models stay close, in a suitable norm, on
  finite time intervals. We provide numerical evidence of our perturbative results.
\end{abstract}

\section{Introduction} The dynamics of brain activity is
intrinsically nonlocal: at microscopic level, the cellular nucleus of a neuron is in
its \textit{soma}, whose average diameter is 20 micrometres; its \textit{dendritic
branches}, collecting input from other neurons, have an extent variable between 100
micrometers and 1 millimetre \cite[Table 1.1]{Harris}; its \changed{\textit{axons}}, which
transmit activity to other neurons, have terminations as far as 1 metre away from the
soma \cite{maChapterAxonGrowth2013}. 

Coarse-grained descriptions of large-scale neuronal activity inherit this
nonlocality, and model long-distance
cortical connections via integral terms. The simplest and most popular of these
models is the \textit{neural field equation} \begin{equation} \label{1old}
  \begin{aligned} & \partial _{t}v(x,t) = -\gamma v(x ,t) + \int_{\Omega}
  W(x,x^{\prime })S(v(x^{\prime}))dx^{\prime } + G(x,t), && (x,t) \in \Omega \times
[0,T],\\ & v(x,0) = v_0(x), && x \in \Omega. \end{aligned} \end{equation}

In this model, $v(x,t)$ represents the macroscopic voltage at time $t$ and position
$x$ in the continuum cortex $\Omega$. The voltage decays at rate $\gamma$, and is
influenced by the nonlocal integral coupling, and by the external input $G$. The
former collects contributions from the whole cortex: the \textit{firing rate function}
$S$ acts as a nonlinear gain, and the \textit{synaptic kernel} $W$ models connections
from point $x'$ to point $x$ in the cortex.

Since their introduction by Wilson and Cowan~\cite{wilson1973mathematical}, and by
Amari~\cite{amari1977dynamics}, neural fields have been used to model large-scale
patterns of activity (see the texbooks and monographs
\cite{Ermentrout.1998qno,Ermentrout.2010,Bresloff.2012,
Bressloff.2014k0p,coombes2014neural,coombesNeurodynamics2023}).
Mathematical neuroscientists have progressively developed dynamical systems
methods to study the wide variety of nonlinear patterns supported by neural field
equations, including localised bumps, travelling waves, and rotating waves
\cite{ermentroutExistenceUniquenessTravelling1993,Folias.2005, Gils.2013,
kilpatrick2013wandering, meijerTravellingWavesNeural2014, laing2014numerical,
kilpatrickPulseBifurcationsStochastic2014, Visser.2017,schmidt2020bumps}. The
analytical treatment of these evolution equations was initated in the 1990s, by
proving the existence of travelling wave solutions using perturbative arguments
\cite{ermentroutExistenceUniquenessTravelling1993}. In addition, the well-posedness
of \cref{1old} has been studied with functional analytical methods, as a Cauchy
problem on the spaces of continuous functions \cite{Potthast:2010kb} and square
integrable functions \cite{faugeras2008absolute} defined on the cortical space
$\Omega$.

Since solutions in closed form are available only in simple cases, numerical methods
have been proposed for one- and two-dimensional cortices. These algorithms must
overcome the problem of evaluating efficiently the nonlocal term, and schemes that
leverage rank-reduction \cite{Lima2015xt}, Fast Fourier Transforms over $\Omega$
\cite{Rankin2013} or over $[0,T]$ \cite{Hutt.2010} have been proposed. Recently,
\changed{
collocation and Galerkin, finite elements and spectral schemes 
}
have been
analysed systematically, as abstract projection schemes for neural fields, seen as
Cauchy problems on Banach-spaces \cite{AvitabileProjection}.

The neural field equation \cref{1old} is now a well established phenomenological
model for neurodynamics in the cortex, and has been extended in several ways. Models
with multiple
populations \cite{pintoSpatiallyStructuredActivity2001} and with transmission delays
\cite{faye2010a,meijerTravellingWavesNeural2014} have been proposed in the 2000s, whereas systems incorporating
time-dependent synaptic kernels to model short-term plasticity are becoming available
at the time of writing \cite{cihakMultiscaleMotionDeformation2023}. As new models are conceived, novel
mathematical approaches are developed to analyse them: neural fields with delays, for
instance, led to the development of a dedicated theory based on sun-star calculus
\cite{visserSpikingNeuronsBrain2013,Gils.2013,Visser.2017}.

\changed{
  Models with nonlocal terms arise also in other branches of the life sciences.
  In particular, a wide range of nonlocal reaction-diffusion equations have been
  studied to describe specific aspects of the evolution of biological populations 
  (see \cite[Sections 1.2.4, Chapter 4]{volpertEllipticPartialDifferential2014} and
  references therein).
  In such models integral terms arise naturally when modelling populations competing
  for global resources, speciation, and cooperation in reproduction. In this context,
  a notable example is the nonlocal Fisher-KPP 
  equation~\cite{berestyckiNonlocalFisherKPP2009,perthameConcentrationNonlocalFisher2007}.
  Nonlocal models are also found in models of tumour growth, as a consequence of variations in
  nutrients or oxygen concentration, or to model chemotherapy and immune system
  responses~\cite{carrereInfluenceMutationsPhenotypicallystructured2020}. Recently,
  integro-differential equations have been proposed to model the  emergence of cities
  and urban patterning \cite{whiteley2022modelling}. The growing literature on the
  subject has shown that a rigorous theoretical analysis is paramount to capture the
  nonlocal behaviour of solutions with respect different biological parameters.
  Important investigations of non-local reaction-diffusion theory can be found, for
  instance, in
  \cite{bouinBramsonDelayNonlocal2020,perthameConcentrationNonlocalFisher2007}.
}

The present article concerns the theoretical treatment of one of the more recent
extensions to the neural field equations, which include a diffusive term. Neural
fields with diffusive terms have been analysed in models with delays
\cite{spekNeuralFieldModels2020}: in this context, the powerful machinery of sun-star
calculus has been exploited to tackle the diffusive term. Recently, two models
without delays have been proposed: in the first one, an anisotropic diffusion term
emerges in the voltage equation \cref{1old}, upon modelling dendrites in the tissue
\cite{Avitabile.2020};
in the second one, a neural field for voltage dynamics is coupled to a
reaction-diffusion equation for potassium dynamics, to model spreading depression
\cite{baspinar:hal-04008117}.

The latter models have an interesting and unexplored mathematical structure. The
diffusion-less neural field equation \cref{1old} admits classical
solutions, and the technical apparatus to analyse them relies heavily on the
compactness, boundedness, and Lipschitz properties of the integral operator. In models
with diffusion the scenario changes considerably, owing to the presence of unbounded
differential operators. The motivation behind the present article is that, in
models with diffusion and without delays, one should be able to study solutions using
classical methods for PDE analysis, and view the integral term as a nonlinear,
well-behaved perturbation of a parabolic problem. 

Vanishing diffusion problems are of great importance also outside of mathematical
neuroscience, in the mathematical and physical communities. For instance, such
problems model atmospheric-oceanographic flows and landscape evolution when fluvial erosion
dominates over smoothing tendencies of the soil diffusion; further, artificial
vanishing diffusion is used for stabilization of numerical simulations. An outstanding problem in
fluid mechanics is the control of boundary layer turbulence: small forces of viscous
friction may perceptibly affect the motion of fluids, even in fluids with small
viscosity, such as water and air. Knowledge of the behavior of solutions
for small viscosities is crucial for understanding turbulence phenomena, which have
strong consequences in branches of engineering dealing with car and aircraft
productions, turbine blades, and nano-technology. We refer to
\cite{Chemetov1,Chemetov2,Chemetov3,Chemetov4, Chemetov5} for mathematical studies on
problems with vanishing viscosity.

\begin{figure}
  \centering
  \includegraphics[width=\textwidth]{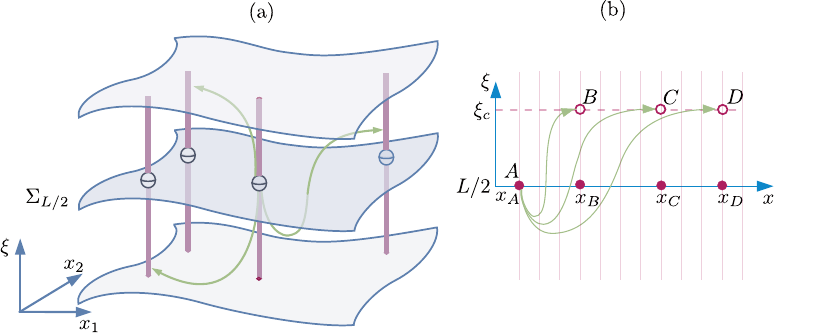}
  \caption{(a) Sketch of the cortical domain $\Omega = \cup_{\xi \in [0,L]}
  \Sigma_\xi$, in which $\Sigma_{L/2}$ is the somatic layer. Dendrites are modelled
  as a continuum of vertical, unbranched fibres, with synaptic connections to other dendrites. (b)
  In \cite{Avitabile.2020}, and in the numerical experiments of the present paper,
  we take $\Sigma_{L/2} \subset \RSet$, with kernel specified by \cref{eq:kernel}.
  This kernel prescribes localised connections from a small neighbourhood of the
  somatic layer $\xi = L/2$, to a neighbourhood of the somatic coordinate $\xi =
  \xi_c$, with strength dependent on the mutual distance between points, projected on
the $x$-axis. This modelling assumption is used in simulations, but not for the
theoretical results, which work for generic kernels, and with somatic layers of
arbitrary dimensions.}
  \label{fig:sketch}
\end{figure}

\subsection{Model with anisotropic diffusion, and main results of the paper}\label{ssec:DiffusionModel} 
\changed{
We demonstrate this strategy
} 
on the model proposed in \cite{Avitabile.2020},
in which the cortex has an embedded anisotropy, owing to its being split into
\textit{somatic} and \textit{dendritic} directions. With reference to the sketch in
\cref{fig:sketch}, the cortex $\Omega \subset \RSet^3$ is composed of 
dendritic fibres, aligned vertically to $U =(0,L) \subset \RSet$, with coordinate
$\xi$. Dendritic fibers are unbranched, and afferent to somas, which form a continuum
somatic layer, a 2-manifold $\Sigma_{L/2} \subset \RSet^3$. The subscript in $\Sigma$
indicates that
the somas belong to a medial layer located at position $\xi = L/2$, and it hints 
that the cortex is conceived as a foliation $\Omega = \cup_{\xi \in U} \Sigma_\xi$ of
which the somatic layer is the leaf by $\xi = L/2$.

In \cite{Avitabile.2020} the somatic layer was assumed to be the 1-torus $\Sigma =
\TSet$, whereas in the present paper we treat the more general case $\Sigma = 
\TSet^n$, for some $n \in \{ 1,2 \}$. The cortex is thus specified as a tensor
product, with coordinates $(x,\xi) \in \Omega = \TSet^n \times U$. Finally, for
notational convenience, we introduce the sets $U_T = U \times [0,T]$, and
$\Omega_T=\Omega \times [0,T]$.

The somato-dendritic model differs from a
standard neural field model, in that the voltage $v(x,\xi,t)$ propagates according to
a cable equation 
along the dendritic fibre direction $\xi$ 
\changed{(see for instance \cite{tuckwellIntroductionTheoreticalNeurobiology2006})}, 
with a 
\changed{
  an external input $G$ and a
}
nonlocal input current $F$ from the somatic layer, 
\begin{equation}\label{eq:NFNew}
  \begin{aligned}
   & 
   \begin{aligned}[b]
   \partial_t v(x,\xi,t) 
   = (-\gamma + \nu \partial_{\xi}^{2})v(x,\xi,t) 
   &+ G(x,\xi,t) \\
   &+F(v(\blank,\blank,t))(x,\xi), 
   \end{aligned}
   && (x,\xi,t) \in \Omega_T, 
   \\
   &\partial_\xi v(x,0,t) = 0, \qquad  \partial_\xi v(x,L,t) = 0, 
   &&(x,t) \in \TSet^n \times [0,T],\\
   &  v(x,\xi,0) = v_0(x,\xi) 
   &&(x,\xi) \in \Omega,
  \end{aligned}
\end{equation}
where periodic boundary conditions in $x$ are tacitly implied by taking $x \in
\TSet^n$, and where the nonlocal term is given by
\begin{equation}\label{eq:FDef}
  F(u)(x,\xi) = \int_{\Omega} W(x,\xi,x',\xi') S(u(x',\xi'))\, dx' d\xi'.
\end{equation}
In passing, we note that Neumann boundary conditions have been chosen to model
no-flux across the leaves at $\xi = 0,L$, but other choices are also possible.

Somatic currents are modelled through the nonlocal operator $F$.
In \cite{Avitabile.2020} the firing rate $S$ is taken to be a bounded monotone
smooth function, in line with neural field literature. In addition, the synaptic
kernel
\begin{equation}\label{eq:kernel}
W(x,\xi,x',\xi') = w(\|x-x'\|_2)\delta_\rho(\xi - \xi_\text{c})\delta_\rho(\xi'-L/2),
\end{equation}
where $w \colon \TSet \to \RSet$, and $\delta_\rho \colon \RSet \to
\RSet$ is a function supported on $(-\rho,\rho)$ for some $\rho >0$, is used to model
currents generated in a neighbourhood of the somatic layer, at
$\xi = L/2$, and transferred to contact points in the neighbourhood of the leaf at
$\xi = \xi_\text{c}$ (see also \cref{fig:sketch}(b)). Concurrently, activity is allowed to propagate within a leaf, via the
distance-dependent function $w$.

Model \cref{eq:NFNew} is anisotropic: instead of the Laplacian operator
$\Delta=\partial^2_{x} + \partial_{\xi}^{2}$, the model features diffusion only along
the fibre coordinate $\xi$. In addition to this local anisotropy, which is not
removable, the synaptic kernel may induce a nonlocal anisotropy, as is the case for
\cref{eq:kernel}. The paper \cite{Avitabile.2020} proposes and analyses a bespoke
numerical scheme for \cref{eq:kernel}, based on differentiation matrices for the
differential operator, and a fast quadrature scheme for the integral operator. This
scheme exploits both anisotropies to evaluate efficiently the right-hand side and
time step the system with an implicit-explicit method. Assuming well-posedness of the
problem, existence and regularity of its solutions, it was shown that \cref{eq:NFNew}
supports anisotropic travelling waves, with propagating speed along the leaves, as
well as Turing-like instabilities.

The present paper provides a functional analytic treatment of \cref{eq:NFNew}, for
generic choices of $S$, $G$, $v_0$ and $W$, that is, without assuming the anisotropic
kernel \cref{eq:kernel}. The main results of the paper can be summarised as follows:
\begin{enumerate}
  \item Under mild, biologically relevant assumptions on $W$ and $S$, the operator
    $F$ is Lipschitz on an appropriately defined function space, and it perturbs
    boundedly the parabolic, anisotropic problem $\partial_t v = (-\gamma +
  \nu\partial_\xi^2)v$ with mixed Neumann/periodic boundary conditions.
   \item The weak form of \cref{eq:NFNew} admits a unique solution, with embedding
     estimates similar to the ones of a nonlinear reaction-diffusion problem with
     local nonlinear forcing.
   \item The analysis of weak solutions $v_\nu$ to \cref{eq:NFNew} for $\nu >0$
     relies on studying weak solutions $v$ to the singular limit $\nu = 0$ of the
     problem, that is, a standard neural field posed on $\Omega$, which does not
     require specifications of  boundary conditions. In so doing, we derive also
     sharper estimates on standard neural field problems, with respect to the ones
     currently available in literature. 
   \item We prove that solutions $v_\nu$ depend continuously on the parameter $\nu$
     and, demanding additional regularity of the problem data, that the
     difference $v_\nu - v$ is an $O(\nu^{1/2})$ as $\nu \to 0$ in an appropriate norm.
   \item We produce numerical evidence of the results above using the kernel and
     numerical scheme studied in \cite{Avitabile.2020}.
\end{enumerate}

The analysis presented here provides the necessary theoretical foundation for
developing numerical schemes on realistic geometries, employing for instance
finite-elements discretisation, or multi-scale schemes for weakly diffusive
problems. Importantly, it seems possible to
adapt the analysis presented below, with minor modifications, to the case of generic
$\Omega$ with curved leaves $\Sigma_\xi$, or to neural fields with multiple
populations and isotropic diffusion, such as the ones in \cite{baspinar:hal-04008117}.

The paper is structured as follows: in \cref{sec:preliminary} we set-up the abstract
formulation of system \cref{eq:NFNew}; we study the singular problem ($\nu=0$) in
\cref{sec:nu0}, and the regular problem ($\nu >0$) in \cref{sec:nuGeq0}; we prove
continuous dependence of the solutions on $\nu$ in \cref{sec:contDepend}, and the
$O(\nu^{1/2})$ scaling in \cref{sec:convergenceRate}; we present numerical
experiments in \cref{sec:numerics}, and we conclude in \cref{sec:conclusions}.

\section{Preliminary assumptions and results} \label{sec:preliminary}
We begin by describing the appropriate function spaces that will be used in
the analysis of our problem.

For a given topological space $\mathbb{X}$, we denote by $L^{2}(\mathbb{X})$ the
space of measurable square integrable functions $u:\mathbb{X}%
\rightarrow \mathbb{R}$ with the usual inner product $(\blank ,\blank )$ and
the norm $\Vert u\Vert _{L^{2}(\mathbb{X})}=\sqrt{(u,u)}$.

Let $X$ be a real Banach space with norm $\left\Vert \blank \right\Vert
_{X}.$ We shall denote by $X^*$ the dual of $X$, and by $L^{q}(0,T;X)$, $q \in
[0,\infty ]$, the space of $X$-valued measurable functions endowed with the norm 
\[
\begin{aligned}
& \Vert u\Vert _{L^{q}(0,T;X)} 
  = \biggl( \int_{0}^{T}\Vert u(t)\Vert_{X}^{q}\biggr) ^{1/q}<\infty , 
  && 1\leq q<\infty, \\
& \Vert u\Vert _{L^{\infty }(0,T;X)} =\sup_{t\in \lbrack 0,T]}\Vert u(t)\Vert _{X}<\infty, 
  && q=\infty.
\end{aligned}
\]

We make the following standing assumptions: 
\begin{hypothesis}[General assumptions]\label{hyp:general}
  \begin{enumerate}
    \item The cortical domain is $\Omega = \TSet^n \times U$, where $U \in
      (0,L)$, $L>0$.
    \item The synaptic kernel $W$ is a function in $L^2(\Omega \times
      \Omega)$.
    \item The firing rate $S \colon \RSet \to \RSet$ is a bounded and everywhere differentiable Lipschitz
      function. %$in $C_b(\RSet)$.
  \end{enumerate}
\end{hypothesis}
Further regularity assumptions on the forcing $G$ and the initial condition $v_0$ will be
provided in the statements of the upcoming results. In preparation for the abstract
setup of the problem we derive useful properties of the integral operator
\begin{equation}\label{eq:integral}
  F(u)(x,\xi) = \int_{\TSet^n \times U} W(x,\xi,x',\xi') S(u(x',\xi'))\, dx' d\xi'.
\end{equation}

\begin{lemma}[Boundedness an Lipschitzianity of $F$]\label{cor:FEstimates}
  Under \cref{hyp:general} the operator $F$ defined in \cref{eq:integral} is
  on $L^2(\Omega)$ to itself, and there exists a constant $K_F >0 $ such that, for
  any $u, v \in L^2(\Omega)$
  \[
    \| F(u) \|_{L^2(\Omega)} \leq K_F,  
    \qquad 
    \| F(u) - F(v) \|_{L^2(\Omega)} \leq K_F \| u - v \|_{L^2(\Omega)}.
  \]
  
  Further for any $q \in [1,+\infty]$, any $u,v \in L^q(0,T;L^2(\Omega))$, and any $t \in [0,T]$
  \[
    \| F(u(t)) \|_{L^2(\Omega)} \leq K_F,  
    \qquad 
    \| F(u(t)) - F(v(t)) \|_{L^2(\Omega)} \leq K_F \| u(t) - v(t) \|_{L^2(\Omega)}.
  \]
\end{lemma}
\begin{proof}
  The results follow standard arguments, summarised in \cite{AvitabileProjection},
  and references therein, in particular \cite{faugeras2008absolute}. Denote by $N$
  the Nemytskii operator $N(u)(x) = S(u(x))$, and by $A$ the linear integral operator
  with kernel $W$, so as to write $F(u) = (A \circ N)(u)$. 

  By \cite[Lemma 2.5]{AvitabileProjection} the operator $N \colon L^2(\Omega) \to
  L^2(\Omega)$ satisfies the homogeneous bound $\| N(u) \|_{L^2(\Omega)}  \leq
  |\Omega|^{1/2} \|  S \|_{\infty }$, and is Lipschitz with constant $\|S'
  \|_{\infty}$. Further by \cite[Lemma 2.6]{AvitabileProjection}, $A \colon
  L^2(\Omega) \to L^2(\Omega)$ is bounded with $\| A \|  \leq \| w \|_{L^2(\Omega \times \Omega)}$. 
  
  Putting this together, it holds for all $u \in L^2(\Omega)$
  \[
    \| F(u) \|_{L^2(\Omega)} \leq \| A \| \| N(u) \|_{L^2(\Omega)} 
      \leq \| w \|_{L^2(\Omega \times \Omega)} |\Omega|^{1/2} \|  S \|_{\infty},
  \]
  and for all $u,v \in L^2(\Omega)$
\[
  \begin{aligned}
    \| F(u) - F(v) \|_{L^2(\Omega)} 
        \leq  \| A \| \| S(u)-S(v) \|_{L^2(\Omega)} 
        \leq  \| w \|_{L^2(\Omega \times \Omega)} \| S' \|_{\infty} \| u-v \|_{L^2(\Omega)}.
  \end{aligned}
\]
The statement is proved by setting
\[
  K_F = \| w \|_{L^2(\Omega \times \Omega)} \max \bigl( |\Omega|^{1/2} \|  S \|_{\infty}, \| S' \|_{\infty} \bigr),
\]
and noting that any function in $L^q(0,T;L^2(\Omega))$ takes values in $L^2(\Omega)$.
\end{proof}

\subsection{Abstract problem set-up}\label{ssec:abstractProblems} 
Equipped with the operator $F$ of the previous sections, we now consider two abstract
problems. The first one is a rewriting of \cref{eq:NFNew}, the neural field problem
with anisotropic diffusion posed on $\Omega = \TSet^n \times U$,
\begin{equation}\label{eq:model}
  \begin{aligned}
    & \partial_t v = \nu \partial_\xi^2 v - \gamma v + F(v) + G 
    && \text{on $\TSet^n \times U \times (0,T]$,}
    \\
    & \partial_\xi v = 0
    && \text{on $\TSet^n \times \partial U \times (0,T],$}
    \\
    & v = v_0
    && \text{on $\TSet^n \times U \times \{t = 0\}$.}
  \end{aligned}
\end{equation}

The second one is a Cauchy problem obtained from \cref{eq:model}
by setting $\nu = 0$ and removing Neumann boundary conditions in $\xi$
\begin{equation}\label{eq:NFAbstract}
  \begin{aligned}
    & \partial_t v = -\gamma v + F(v) + G && \text{on $\TSet^n \times U \times (0,T]$,} \\
    & v = v_0 && \text{on $\TSet^n \times U \times \{t=0\}$.}
  \end{aligned}
\end{equation}
The problem is posed on a tensor product domain between the cortex $\TSet^n$ and the
dendritic segment $U$, but it does not account for dendrites as cables (and there is
no diffusion in the problem). This model is a classical neural field. Existence of
classical solutions to this problem have been studied by several authors from a
functional analytic viewpoint,
\cite{Faugeras:2009gn,Potthast:2010kb,AvitabileProjection}. In the present article we
will discuss weak solutions to the problem, as this provides the appropriate
functional analytical set-up to compare solutions to the problem with anistropic
(dendritic) diffusion.

We will colloquially refer to \cref{eq:model} as the \textit{regular problem} $(\nu >
0)$, and to \cref{eq:NFAbstract} as the \textit{singular problem} ($\nu = 0$). We are
interested in characterising solutions to the former as $\nu$-perturbations of
solutions to the latter.

\section{Existence of solution to the singular problem, 
  \texorpdfstring{$\nu=0$}{}}\label{sec:nu0}
Our
starting point is a result descending from the classical theory described in
\cite{Evans,Mih}, and concerning the solution of a linear
ordinary differential equation in the Sobolev space $H^{1}(0,T;\RSet)$. The results
presented in this section offer an accessible entry point to the ones of the
following sections, with respect to which they have a similar approach, but fewer
technical details.

\begin{lemma}
\label{the_1 copy(1)} For any $\gamma \in \RSet_{\geq 0}$, $g\in \mathbb{R}$, and $f\in
L^{2}(0,T;\RSet)$ there exists a unique solution $v\in H^{1}(0,T;\RSet)$ to the
linear inhomogeneous
ordinary differential equation
\begin{equation}
\begin{aligned}
& \frac{d}{dt} v=-\gamma v+f  && \text{a.e. in $(0,T)$}, \\ 
& v=g                      &&  \text{on $\{ t =0 \}$}.
\end{aligned}
 \label{nu0}
\end{equation}
Further, the solution $v$ satisfies the following estimate
\begin{equation*}
\Vert v\Vert _{L^{\infty }(0,T;\RSet)}^{2}+\gamma ||v\Vert _{L^{2}(0,T;\RSet)}^{2}+\Vert
\partial _{t}v\Vert _{L^{2}(0,T;\RSet)}^{2}\leq C\left( |g|+\Vert
f\Vert_{L^{2}(0,T;\RSet)}^{2}\right),
\end{equation*}%
where the positive constant $C$ depends solely on $T$.
\end{lemma}

We now prove a statement in a similar spirit of the previous lemma, but for the
abstract singular neural field problem.

\begin{theorem} \label{the_2 copy(1)}
 Under \cref{hyp:general}, for any $\gamma >0$, $v_0 \in
 L^{2}(\Omega)$ and $G\in L^{2}(\Omega_T)$, there exists a unique solution 
  \begin{equation}\label{eq:emb_1Nik}
      v \in L^2(\Omega_T) 
      \cap H^{1}(0,T;L^2(\Omega))  
      \hookrightarrow C([0,T];L^2(\Omega)),
  \end{equation}
to the problem 
\begin{equation} \label{nu3}
  \begin{aligned}
    & \partial_t v = -\gamma v + F(v) + G && \text{a.e. in $\Omega_T$,} \\
    & v = v_0 && \text{on $\Omega \times \{t=0\}$.}
  \end{aligned}
\end{equation}
such that
\begin{equation}\label{eq:nu3Est}
  \Vert v\Vert _{L^{\infty }(0,T;L^{2}(\Omega ))}^{2}+\gamma \Vert v\Vert
  _{L^{2}(\Omega _{T})}^{2}+\Vert \partial _{t}v\Vert _{L^{2}(\Omega
  _{T})}^{2}\leq C_{0}^{\prime }  \notag
\end{equation}
for some constant $C_{0}^{\prime }$ which depends only on $\Vert v_{0}\Vert
_{L^{2}(\Omega )},$ $\Vert G\Vert _{L^{2}(\Omega _{T})}$ and $K_{F}$ defined in
\cref{cor:FEstimates}
\end{theorem}

\begin{proof}
In order to prove the solvability of problem \cref{nu3} we use  Banach's fixed point
theorem \cite{Evans}. For some constants $\tau \in [0,T]$, and $M_\tau > 0$ let us
define the set 
\begin{equation}
\mathcal{B}_{\tau}=\{u\in L^{\infty }(0,\tau;L^{2}(\Omega )) \colon \Vert
u\Vert _{L^{\infty }(0,\tau;L^{2}(\Omega ))}\leq M_\tau\}  \label{nu4}
\end{equation}
and consider the operator
\[
  P \colon L^{\infty }(0,\tau;L^{2}(\Omega))\rightarrow L^{\infty }(0,\tau;L^{2}(\Omega
  )), \qquad u \mapsto v,
\]
where $v=v(x,\xi ,t)$ is the solution of the ODE problem \cref{nu0} with the
right hand side $f$ and the initial data $g$, defined by 
\begin{equation}\label{fg0}
f(x,\xi ,t)=G(x,\xi ,t)+F(u)(x,\xi ,t), \qquad  g(x,\xi
)=v_{0}(x,\xi )\quad \text{for }(x,\xi ,t)\in \Omega _{\tau}. 
\end{equation}
It is important to stress that the ODE so defined has solutions $t \mapsto v(x,\xi,t)$ in which
$(x,\xi)$ are mere parameters, and in which $u$ is an input. Further, fixed points of
$P$ are solutions to the singular neural field equation \cref{eq:NFAbstract}.

We first show that $P \colon \calB_\tau \to \calB_\tau$, is a contraction provided $\tau$ and $M_\tau$
are suitably chosen. Pick $\tau$, $M_\tau$, $u \in \calB_\tau$.
Applying \cref{the_1 copy(1)} to the ODE \cref{nu0} with functions $f$ and $g$
 specified in \cref{fg0}, we deduce the existence of a unique solution $v=v(x,\xi
 ,t)$ of (\ref{nu0}) dependent on parameters $(x,\xi )\in \Omega $, such that, for
 almost every $(x,\xi )\in \Omega$ 
 \[
 \begin{aligned}
  \Vert v(x,\xi )\Vert _{L^{\infty }(0,\tau)}^{2} 
  &+\gamma \Vert v(x,\xi )\Vert _{L^{2}(0,\tau)}^{2} 
  +\Vert \partial _{t}v(x,\xi )\Vert _{L^{2}(0,\tau)}^{2} \\
 &\leq C\left( |v_{0}(x,\xi )|^{2}+\Vert G(x,\xi )\Vert
 _{L^{2}(0,\tau)}^{2}+\Vert F(u)(x,\xi )\Vert _{L^{2}(0,\tau)}^{2}\right). \qquad 
 \end{aligned}%
 \]
 Integrating this inequality over $(x,\xi )\in \Omega $, using
 \cref{cor:FEstimates} we obtain
 \[
 \begin{aligned}
 \Vert v\Vert _{L^{\infty }(0,\tau;L^{2}(\Omega ))}^{2}
     +\gamma \Vert v\Vert _{L^{2}(\Omega _{\tau})}^{2} 
    & +\Vert \partial _{t}v\Vert _{L^{2}(\Omega _{\tau})}^{2} \\
    & \leq C(\Vert v_{0}\Vert _{L^{2}(\Omega )}^{2} 
    + \Vert G\Vert _{L^{2}(\Omega _{\tau})}^{2}+K^2_{F})=:C_{0}^{\prime }.
  \end{aligned}%
 \]
 Therefore setting $M_\tau=C_{0}^{\prime }$ in (\ref{nu4}) we have $P \colon \calB_\tau
 \to \calB_\tau$. We then show that $P$ is a contraction
 for small enough $\tau$. Let $u_{1}, u_{2}\in \mathcal{B}_{\tau}$ and
 $v_{i}=P(u_{i}),$ $i=1,2,$ be solutions
of \cref{nu0} with right hand side $f_{i}$ and initial data $g$ defined by
\[
  f_{i}(x,\xi ,t)=G(x,\xi ,t)+F(u_{i})(x,\xi ,t),\qquad 
  g(x,\xi )=v_{0}(x,\xi ),
  \qquad (x,\xi ,t)\in \Omega _{\tau},
\]
respectively. Since $v_{1}$ and $v_{2}$ satisfy the ODE \cref{nu0}, with data defined
above, the difference $z=v_{1}-v_{2}$ satisfies the inequality
\[
  \begin{aligned}
    \frac{d}{dt}|z|^{2} +\gamma |z|^{2}=(F(u_{1})&-F(u_{2}))z \\
       & \leq \frac{1}{2\gamma}|F(u_{1})-F(u_{2})|^{2}+\frac{\gamma }{2}|z|^2 
  \qquad \text{for a.e. $(x,\xi ,t) \in \Omega _{\tau}$.}
  \end{aligned}
\]
In the last passage, we have multiplied the ODE by $z$, and used the inequality $2\alpha 
\beta = 2(\alpha \epsi) (\beta/\epsi)\leq (\alpha\epsi)^2 + (\beta/\epsi)^2$ for
$\alpha,\beta,\epsi \in \RSet$, with
$\gamma = 1/\epsi^{2}$. We now neglect the terms proportional to $\gamma z^2$,
integrate over $\Omega$ with respect to $(x,\xi)$, over $(0,\tau)$ with respect to
$t$, and use \cref{cor:FEstimates} to obtain
\begin{equation*}
\Vert z(t)\Vert _{L^{2}(\Omega )}^{2}
\leq C_{\ast }\int_{0}^{t} \Vert u_1(s) - u_2(s)\Vert_{L^{2}(\Omega )}^{2}ds
\leq C_{\ast }t\Vert u_1 - u_2 \Vert _{L^{\infty}(0,T;L^{2}(\Omega ))}^{2},\quad t\in
(0,\tau),
\end{equation*}
for some constant $C_*$. Selecting $\tau$ such that $\tau C_{\ast }<1,$ we obtain 
\[
\| P(u_1) - P(u_2) \|_{L^\infty(0,\tau;L^2(\Omega))} 
  = \| z \|_{L^\infty(0,\tau;L^2(\Omega))}  
  <  \| u_1 - u_2 \|_{L^\infty(0,\tau;L^2(\Omega))},
\]
hence $P \colon \calB_\tau \to \calB_\tau$ is a contraction and by Banach's fixed
point theorem admits a fixed point $v=P(v)$, that is, the solution of \cref{nu3} on the
interval $[0,\tau]$. Studying the differential equation \cref{nu3} on the time
interval $[\tau,2\tau]$, with the condition $v_0 = v(\tau)$ one proves the existence
of a solution of system \cref{nu3} on the time interval $[\tau,2\tau]$. This
process can be continued to extend the solution on $[0,T]$, satisfying the estimate
\cref{eq:nu3Est}. 

Finally, we show uniqueness of the solution. If there exist two different solutions
$v_{1}$ and $v_{2},$ then the difference $v=v_{1}-v_{2}$ satisfies the problem
\[
\begin{aligned}
 &\partial _{t}v =-\gamma v+F(v_{1})-F(v_{1}) && \text{on $\Omega_{T} = \Omega \times (0,T]$}, \\
 & v=0 && \text{on $\Omega \times \{ t = 0 \}$},
\end{aligned}%
\]
Hence we have 
\[
\frac{d}{dt}|v|^{2} + \gamma |v|^{2} = (F(v_{1})-F(v_{2}))v \leq
\frac{1}{2\gamma}|F(v_{1})-F(v_{2})|^{2}+\frac{\gamma }{2}|v|^2, \qquad  \text{a.e. $(x,\xi,t)\in \Omega _{T}$}.
\]
Integrating and using \cref{cor:FEstimates} we deduce
\begin{equation*}
\Vert v(t) \Vert _{L^{2}(\Omega )}^{2}\leq C \int_{0}^{t} \Vert v(s)\Vert
_{L^{2}(\Omega )}^{2}ds\quad \text{for }t\in (0,T),
\end{equation*}
and Gronwall's inequality gives $v(t)=v_{1}(t)-v_{2}(t)\equiv 0$. 
\end{proof}

\section{Existence of solution to the regular problem \texorpdfstring{$\nu>0$}{}}
\label{sec:nuGeq0}

To make progress towards studying solutions of the regular problem, we present a
result from the theory of parabolic equations. Consider the following auxiliary
linear inhomogeneous problem for unknown functions $z(x,\xi,t)$, with $z \colon \Omega
\times [0,T] \to \RSet$,
\begin{equation}\label{heat}
  \begin{aligned}
  & \partial_t z = \epsi \partial_x^2 z + \nu \partial_{\xi}^{2} z - \gamma z + f 
    && \text{on $\TSet^n \times U \times (0,T]$,} \\
  & \partial_\xi z = 0 
  && \text{on $\TSet^n \times \partial U \times [0,T]$,} \\
  & z = g
    && \text{on $\TSet^n \times U    \times \{ t = 0 \}$}. \\
  \end{aligned}
\end{equation}

The following solvability result is obtained by the classical theory presented in
\cite[page 378, Theorem 3] {Evans}, \cite[page 372, Theorem 3]{Mih}, and we omit its
proof for brevity. 
\begin{lemma} \label{the_1}
  Fix $\epsi, \nu, \gamma>0$, $g\in L^{2}(\Omega)$ and $f\in L^{2}(\Omega_{T})$. The weak
  formulation of the linear inhomogeneous, anisotropic heat equation \cref{heat}, 
  \begin{equation} \label{eq:heatWeak}
  \begin{aligned}
   & 
   \begin{aligned}[b]
   \langle \partial _{t}z,\varphi \rangle 
   + \epsi (\partial _{x }z,\partial _{x}\varphi) 
   &+ \nu (\partial _{\xi }z,\partial _{\xi }\varphi) \\
   &+\gamma (z,\varphi ) 
   = (f,\varphi), 
      \quad
   \varphi \in H^{1}(\Omega), 
   \end{aligned}
  &&\textrm{a.e. in $(0,T)$}, 
\\
   & z(0) = g, 
   &&\text{on $U$}, 
  \end{aligned}
  \end{equation} 
  where $\langle \psi,\varphi \rangle$ denotes the duality pairing between $\psi \in
   H^{1}(\Omega)^{\ast }$ and $\varphi \in H^{1}(\Omega)$, admits a
  unique solution satisfying
  \[
  z \in L^{2}(0,T;H^{1}(\Omega))\cap H^{1}(0,T; H^{1}(\Omega)^{\ast})\hookrightarrow C([0,T];L^{2}(\Omega)),
  \]
  and the estimate
  \[ \label{embeding1}
  \begin{aligned}
  \Vert z\Vert_{L^{\infty }(0,T;L^{2}(\Omega))}^{2}
  &+\epsi \Vert \partial _{x }z\Vert _{L^{2}(\Omega_{T})}^{2}
  +\nu \Vert \partial _{\xi }z\Vert _{L^{2}(\Omega_{T})}^{2}
  \\
  &+\gamma \Vert z\Vert _{L^{2}(\Omega_{T})}^{2}
  +\Vert \partial _{t}z\Vert _{L^{2}(0,T; H^{1}(\Omega)^{\ast })}^{2} 
  \leq C\left( \Vert g \Vert _{L^{2}(\Omega)}^{2}+\Vert f\Vert _{L^{2}(\Omega_{T})}^{2}\right) 
  \end{aligned}%
  \]
  where $C >0$ is a constant depending solely on $T$ and on the size $L$ of the
  interval $U$.
\end{lemma}

The proof of existence and uniqueness of weak solutions to \cref{eq:model}, relies on
the following auxiliary  result, which characterizes weak solutions to problem
\cref{heat} with $\epsi=0$, in which the dependence on $x$ in the forcing $f$ and
initial condition $g$ is retained.

\begin{lemma} \label{the_1_x} 
  For any given $\nu,\gamma >0$, $g\in L^{2}(\Omega)$ and $f\in L^{2}(\Omega_T)$
  there exists a unique $v$ satisfying the $\epsi =0$ weak form of \cref{heat}, namely
  \begin{equation}\label{eq:heatWeakX}
  \begin{aligned}
   & 
   \begin{aligned}[b]
   \langle \partial _{t}v,\varphi \rangle 
    + \nu (\partial _{\xi }v,\partial _{\xi }\varphi) 
   &+ \gamma (v,\varphi ) \\
   &= (f,\varphi ), 
      \quad \varphi \in L^2(\TSet^n) \times H^{1}(U), 
  \end{aligned}
   &&\textrm{a.e. in $(0,T)$}, 
   \\
   & z(0) = g 
   &&\text{on $\TSet^n \times U$}, 
  \end{aligned}%
  \end{equation}
  where $\langle \psi,\varphi \rangle$ denotes the duality pairing between $\psi \in 
  L^2(\TSet^n) \times \left(H^{1}(U)\right)^{\ast}$ and $\varphi \in L^2(\TSet^n) \times H^{1}(U)$.
  In addition
  \begin{equation}\label{eq:emb_1}
    \begin{aligned}
      v \in L^{2}(0,T;L^2(\TSet^n) \times H^1(U))) 
      \cap H^{1}(0,T;L^2(\TSet^n) 
      & \times H^1(U)^*)  \\ 
      & \hookrightarrow C([0,T];L^2(\Omega)),
    \end{aligned}
\end{equation}
  and the following estimate holds
  \begin{equation}\label{embeding1X} 
  \begin{aligned}
  \Vert v\Vert _{L^{\infty }(0,T;L^{2}(\Omega))}^{2} 
  &+\nu \Vert \partial _{\xi }v\Vert _{L^{2}(\Omega_{T})}^{2}
    +\gamma \Vert v\Vert _{L^{2}(\Omega_{T})}^{2}
    +\Vert \partial _{t}v\Vert _{L^{2}(0,T;L^2(\TSet)^n \times  H^{1}(U)^{\ast })}^{2} 
    \\
  &\leq C\left( \Vert g \Vert _{L^{2}(\Omega)}^{2}+\Vert f\Vert _{L^{2}(\Omega_\tau)}^{2}\right),
  \end{aligned}
  \end{equation}
  where $C >0$ is a constant depending solely on $T$, and on the size $L$ of the interval $U$.
\end{lemma}
\begin{proof}
  For fixed $\epsi>0$, let $v_{\epsi}$ be the unique weak solution to
  \cref{eq:heatWeak}. The statement is proved by studying the
  $\epsi \to 0$ limit of $v_\epsi$. In passing, we note that the constant $C$ in the
  estimate of \cref{the_1} is independent of $\epsi$ (and $\nu$), and that the initial
  condition is well defined by the embedding results presented in \cite[Theorems 2,
  3, pages 302--303]{Evans}. Owing to the estimate in \cref{the_1}, there exists a
  sequence $\{ v_\epsi \}_{\epsi > 0}$ and a function $v$ such that
  \[
    \begin{aligned}
      v_\epsi & \wto v 
      && \text{$*$-weakly in $L^\infty(0,T;L^2(\Omega))$ as $\epsi \to 0$,} \\
      v_\epsi, \epsi\partial_x v_\epsi, \partial_\xi v_\epsi & \wto v_\epsi, 0, \partial_\xi v 
      &&\text{weakly in $L^2(\Omega_T)$ as $\epsi \to 0$,} \\
      \partial_t v_\epsi & \wto \partial_t v
      && \text{weakly in $L^2(0,T;H^1(\Omega)^*)$ as $\epsi \to 0$,}
    \end{aligned}
  \]
  respectively, and using these convergence results we obtain that the limit function
  $v$ satisfies \cref{eq:heatWeakX}. Moreover the weak semi-continuity property of 
  integral
 and the estimate in \cref{the_1} give that $v$ satisfies \cref{embeding1X}.
  Finally, since $v$ is the unique solution to \cref{eq:heatWeakX}, we deduce that
  $v$ has the regularity \cref{embeding1X}. 
\end{proof}

After these preliminaries, we can address the well-posedness of the weak problem for the
neural field with anisotropic diffusion, system \cref{eq:model}, when $ \nu > 0$.
\begin{theorem}[Weak solution, $\nu > 0$]\label{the_2}
 Assume \cref{hyp:general}, and fix $\nu >0$, $\gamma \geq 0$,
 $v_0 \in L^2(\Omega)$, and $G \in L^2(\Omega_T)$. 
  There exists a unique $v$ satisfying the regularity condition \cref{eq:emb_1} such
  that
  \begin{equation}\label{eq:NFWeak}
  \begin{aligned}
   & 
   \begin{aligned}[b]
   \langle \partial _{t}v,\varphi \rangle +
     & \nu (\partial _{\xi }v,\partial _{\xi }\varphi)
      +\gamma (v,\varphi ) \\
     &= (F(v)+G,\varphi ), 
     \qquad 
     \varphi \in L^{2}(\TSet^n) \times H^{1}(U), 
   \end{aligned}
   && \textrm{a.e. in $(0,T)$},
    \\
   & v(0) =v_0,
   &&\text{on $\TSet^n \times U$}.
  \end{aligned}
  \end{equation}
  In addition, there exists a constant $C_0$ depending on $\| G \|_{L^2(\Omega_T)}$,
  $\| v_0 \|_{L^2(\TSet^n \times U)}$, and on the constant $K_F$ defined in
  \cref{cor:FEstimates}, such that
  \[\label{est}
  \Vert v\Vert _{L^{\infty }(0,T;L^{2}(\Omega ))}^{2}
  +\nu \Vert \partial _{\xi} v \Vert _{L^{2}(\Omega _{T})}^{2}
  +\gamma \Vert v\Vert _{L^{2}(\Omega_{T})}^{2}
  +\Vert \partial _{t}v\Vert _{L^{2}(0,T;L^{2}(\mathbb{T}^{n})\times H^{1}(U)^{\ast })}^{2}\leq C_{0}.
  \]
\end{theorem}

\begin{proof}
  The structure of this proof resembles the one of \cref{the_2 copy(1)}: we use a
  local-in-time Banach fixed point argument, which we bootstrap to extend  the
  solution on
  $[0,T]$. Consider the set
  \begin{equation}  \label{eq:Ball}
    \calB_\tau = \{ u\in L^{\infty }(0,\tau;L^{2}(\Omega)) \colon 
                     \Vert u\Vert _{L^{\infty }(0,\tau;L^{2}(\Omega ))}\leq M_\tau\}, 
  \end{equation}
  that is, the closed ball of radius $M_\tau$ centered at the origin in $L^{\infty
  }(0,\tau;L^{2}(\Omega))$, where $\tau \in [0,T]$ and $M_\tau$ will be specified
  later on. We define the operator
\[
    P \colon L^{\infty }(0,\tau;L^{2}(\Omega ))\rightarrow L^{\infty
    }(0,\tau;L^{2}(\Omega )), \qquad u \mapsto v,
\]
where $v$ is the solution to the weak problem \cref{eq:heatWeakX} in \cref{the_1_x}
with forcing and initial conditions 
\begin{equation}\label{eq:fDef}
  f(x,\xi ,t) = F(u)(x,\xi,t) + G(x,\xi,t),
  \qquad 
  g(x,\xi) = v_0(x,\xi), 
  \qquad 
  (x,\xi,t) \in \Omega_\tau.
\end{equation}

We claim the existence of $\tau \in [0,T]$ and of $M_\tau > 0$ such that $P \colon
\calB_\tau \to \calB_\tau$ is a contraction which in turn, by Banach's fixed point
theorem, implies the existence on the time interval $(0,\tau)$ of a solution to
\cref{eq:NFWeak}, that is, a weak
solution to the nonlinear neural field problem with anisotropic diffusion
\cref{eq:model} with $\nu > 0$.

Firstly, we select $M_\tau$ so that $P \colon \calB_\tau \to \calB_\tau$ for any
$\tau \in [0,T]$. \Cref{the_1_x} ensures $P$ is injective. The estimate
\cref{embeding1X}, of which we
neglect the second, third, and fourth term on the left-hand side, guarantees that $P$
is indeed on $L^{\infty }(0,\tau;L^{2}(\Omega))$ to itself.
The choice 
\[
M_\tau =  C
    \bigg[ 
      \| v_0 \|^2_{L^2(\Omega)}
      + \| G \|^2_{L^2(\Omega_\tau)}
      + \tau^2 K_{F}^2
    \bigg],
\]
guarantees also that $\| P(u) \|_{L^{\infty }(0,\tau;L^{2}(\Omega))} \leq M_\tau$,
hence $P \colon \calB_\tau \to \calB_\tau$ for any $\tau \in [0,T]$.

Secondly, we show that $P$ is a contraction, that is, we prove that, for sufficiently
small $\tau$, we can find a constant $K_* \in [0,1)$ such that
\[
  \| P(u_1) - P(u_2) \|_{L^{\infty }(0,\tau;L^{2}(\Omega ))} \leq K_*
  \| u_1 - u_2 \|_{L^{\infty }(0,\tau;L^{2}(\Omega ))} 
  \qquad u_1, u_2 \in \calB_\tau
\]
To this end, fix $u_1, u_2 \in \calB_\tau$, and let $v_1=P(u_1)$, $v_2=P(u_2)$, that
is, the unique solutions to the weak problem in \cref{the_1_x}, with 
\[
  f_i(x,\xi ,t) = F(u_i)(x,\xi,t) + G(x,\xi,t), \qquad i=1,2, \qquad (x,\xi,t) \in \Omega_\tau,
\]
while keeping all other data constant.
The function $z = v_1 - v_2$ satisfies, for any test function $\varphi \in
L^2(\TSet^n) \times H^{1}(U)$
\[
  \begin{aligned}
    \langle \partial _{t} z,\varphi \rangle +
     \nu (\partial _{\xi } z,\partial _{\xi }\varphi)+\gamma (z,\varphi ) 
   &  = (F(u_1)-F(u_2),\varphi ), 
   & & \textrm{ a.e. $(0,\tau)$}, \\
    z(0) &=0, & &\text{on $\Omega$}.
  \end{aligned}
\]

We choose $\phi = z(t)$, recall that 
\begin{equation}\label{eq:vCont}
  \langle \partial_t z, z \rangle = \frac{d}{dt} \| z \|^2_{L^2(\Omega)}
\qquad \text{a.e. in $(0,\tau)$}, \qquad z \in C([0,\tau];L^2(\Omega)),
\end{equation}
and by the result in \cite[page 35]{mnrr} we have
\[
  \begin{aligned}
    \frac{d}{dt}\| z \|^2_{L^2(\Omega)} 
  & + \nu \| \partial_\xi z \|^2_{L^2(\Omega)} 
   + \gamma \| z \|^2_{L^2(\Omega)}  \\
  & 
  \leq 
    \frac{1}{2\gamma} \| F(u_1) - F(u_2)\|^2_{L^2(\Omega)} 
    + \frac{\gamma}{2} \| z \|^2_{L^2(\Omega)}  
    \quad \text{a.e. in $(0,\tau)$.} \\
  \end{aligned}
\]
Neglecting the second and third terms on the left-hand side, and using that  $F$ is Lipschitz function, we arrive at
\[
  \begin{aligned}
    \| v_1(t) - v_2(t) \|^2_{L^2(\Omega)} 
    & \leq C_* \int_0^t \| u_1(s) - u_2(s) \|^2_{L^2(\Omega)}\, ds  \\
    & \leq C_* t \| u_1 - u_2 \|^2_{L^\infty(0,\tau;L^2(\Omega))}
    \qquad \text{for all $t \in (0,\tau)$,}
  \end{aligned}
\]
for some positive constant $C_*$ depending on $K_F$, as defined in
\cref{cor:FEstimates}.
Therefore, picking $\tau < 1/C_*$ we obtain
\[
  \| P(u_1) - P(u_2) \|_{L^{\infty }(0,\tau;L^{2}(\Omega ))} < 
  \| u_1 - u_2 \|_{L^{\infty }(0,\tau;L^{2}(\Omega ))} 
  \qquad u_1, u_2 \in \calB_\tau.
\]

We have thus proven the existence of a unique solution to \cref{eq:NFWeak} on
$[0,\tau]$. The function $v$ is continuous on the time variable with values in $L^2(\Omega)$ by \cref{eq:emb_1}, and can be extended to
a continuous function on $[0,2\tau]$ by repeating the steps above on the
initial-value problem \cref{eq:NFWeak} posed on $(\tau,2\tau)$ with initial condition
$v(x,\xi,\tau)$. Iterating this process we find a continuous solution on the whole
$[0,T]$. It remains to show that the resulting solution $v$ is unique. If two
solutions $v_1, v_2$ then $v=v_1-v_2$ satisfies the equality \cref{eq:NFWeak} with
$F(v)=F(v_1) - F(v_2)$, initial condition $v_0 =0$, and $G=0$. Choosing
$\phi=v(t)$ in this equality leads to 
\[
\begin{aligned}
\frac{d}{dt} \| v \| _{L^{2}(\Omega )}^{2} 
  &+\nu \Vert \partial _{\xi }v\Vert_{L^{2}(\Omega )}^{2}
  +\gamma \Vert v\Vert_{L^{2}(\Omega)}^{2}=(F(v_{1})-F(v_{2}),v) \\
  &
  \leq 
    \frac{1}{2\gamma} \| F(v_1) - F(v_2)\|^2_{L^2(\Omega)} 
    + \frac{\gamma}{2} \| v \|^2_{L^2(\Omega)}  
    \quad \text{a.e. in $(0,T)$}. \\
\end{aligned}
\]  
Then, integrating this inequality with respect to the time variable over the
interval $(0,t)$ and using \cref{cor:FEstimates}, we obtain 
\begin{equation*}
||v(t)\Vert _{L^{2}(\Omega )}^{2}\leq \tilde C\int_{0}^{t}||v(s)\Vert
_{L^{2}(\Omega )}^{2}ds\quad \text{for a.e. }t\in (0,T).
\end{equation*}
By Gronwall's inequality we conclude $v(t)=v_{1}(t)-v_{2}(t)\equiv 0$.

\changed{
Finally, setting \cref{eq:fDef}
}
on the whole $\Omega_T$, using the
estimate \cref{embeding1X}, and \cref{cor:FEstimates} we derive the embedding
presented in the theorem statement, with 
\[
  C_0 = C \bigl( \|  v_0 \|_{L^2(\Omega)} + \| G \|_{L^2(\Omega_T)} + K_{F}\bigr).
\] 
\end{proof}

\section{Continuous dependence on the diffusion parameter
$\nu$}\label{sec:contDepend} Now that we have studied the well-posedness of the
regular and singular problem, we can address the dependence of solutions on the
diffusion parameter $\nu$. We prove continuous dependence on $\nu$ by expanding both
solutions in a suitable basis of $L^2(\Omega)$, and bounding their difference.

\begin{theorem} \label{the_3 convergence}
Assume \cref{hyp:general}, and fix $\gamma \geq 0$,
$v_0 \in L^2(\Omega)$, and $G \in L^2(\Omega_T)$. The solutions $v(t)$ and
$v_{\nu }(t)$ of the singular problem \cref{nu3}, and of the
regular problem \cref{eq:NFWeak}, respectively, satisfy
\begin{equation}\label{eq:Conv}
  \| v - v_\nu \|_{C([0,T];L^{2}(\Omega ))} \to 0
  \qquad \text{as $\nu \to 0$.}
\end{equation}
\end{theorem}
\begin{proof}
  Using  the integral operator $F$ in \cref{eq:integral}, throughout the proof we set 
  \begin{equation}\label{eq:NDef}
    N = G + F(v), \qquad N_\nu = G + F(v_\nu), \qquad \text{a.e in $\Omega_T$}.
  \end{equation}
  By \cref{the_2,the_2
  copy(1)} we have $v,v_{\nu}\in L^{2}(\Omega_{T})$, uniformly on $\nu >0$. The
  main hypothesis $v_0 \in L^2(\Omega)$ and \cref{cor:FEstimates} give
  \[
    \| N \|^2_{L^2(\Omega_T)}, \| N_{\nu} \|^2_{L^2(\Omega_T)} < \infty, 
  \]
  which, combined with Fubini's theorem, imply
  \begin{equation}\label{reg}
    N(x,\blank,\blank), N_{\nu}(x,\blank,\blank) \in L^2(U_T), \qquad
    v_0(x,\blank) \in
    L^2(U) \qquad \text{for a.e. $x \in \TSet^n$.}
  \end{equation}
  
  To demonstrate  \cref{eq:Conv} it is enough to show that for any $\epsi >0$ there exists a function 
  $\nu \colon \RSet_{ \geq 0} \to \RSet_{ \geq 0}$ such that
  \[
  \| v - v_{\nu} \|^2_{C([0,T]; L^2(\Omega))} < \epsi \qquad \text{for all $\nu <
  \nu(\epsi)$.}
  \]
  The proof proceeds in steps, and leads to the norm estimate for $v(t)-v_\nu(t)$
  through various calculations involving point-wise estimates and evaluations
  for fixed $x \in \TSet^n$. We shall sometimes omit the dependence on $x$, to
  simplify the notation.

  \textit{Step 1: eigenvalues and eigenfunctions.} Since $v_\nu(t)$
   satisfies the system \cref{eq:NFWeak}, then, by the classical results \cite{Mih},
   for each $x \in \TSet^n$ we can expand $v_{\nu }$ as a linear combination of
   eigenfunctions of the following spectral problem associated to \cref{eq:NFWeak},
   \begin{equation}\label{eq:eienpairs}
     \begin{aligned}
       \nu \psi'' - \gamma \psi = \lambda \psi, \quad \text{on $U$,}
       \qquad \psi= 0, \quad \text{on $\partial U$.} 
     \end{aligned}
   \end{equation}
   The eigenpairs $\{ (\lambda_k,\psi_k) \}_{k \geq 0}$ of this problem are known
   \[
     \begin{aligned}
     & \psi_0(\xi) = \sqrt{\frac{1}{L}}, 
       && \lambda_0 = -\gamma ,
       && 
       \\
     & \psi_k(\xi) = \sqrt{\frac{2}{L}} \cos \biggl( \frac{k \pi \xi}{L}\biggr),
     && \lambda_k = -\gamma -\nu \biggl( \frac{k \pi \xi}{L}\biggr),
     && k  \geq 1, \\
     \end{aligned}
   \]
   and the set $\{ \psi_k \}_{k \geq 0}$ is an orthogonal basis in $L^2(U)$ with
   respect to the standard inner product and norm
   \[
     (u,z)_{L^2(U)}=\int_{0}^{L}u(\xi )z(\xi )d\xi, \qquad 
     \|u \|^2_{L^{2}(U)}=(u,u)_{L^2(U)},
%%\begin{equation*}
%%((u,z))=
%   \int_{0}^{L}\left\{ \gamma u(\xi )z(\xi )+\nu u^{\prime }(\xi)z^{\prime }(\xi )\right\} d\xi 
%   \quad \text{and }%
%%||u||_{H^{1}(U)}=((u,u))^{1/2}.
%\end{equation*}
   \]
   respectively. Moreover, $\{\psi_{k}\}_{k\geq 0}$ is also the orthogonal basis
   in the space $H^{1}(U)$ with respect to inner product and norm
   \[
     (u,z)_{H_1(U)}=
     \int_{0}^{L}\big[ \gamma u(\xi )z(\xi )+ u'(\xi)z'(\xi)\big] d\xi, \qquad 
     \|u \|^2_{H^1(U)}=(u,u)_{H^1(U)}.
   \]
    
  \textit{Step 2: expansion of $v_\nu$.} We now seek to expand $v_{\nu }(x,\xi ,t)$
  in the form 
  \begin{equation} \label{exp1}
  v_{\nu }(x,\xi ,t)=\sum_{k=0}^{\infty} g_{k}(x,t)\psi _{k}(\xi ),
  \end{equation}
  and we determine coefficients $g_k$ so that $v_\nu$ satisfies
  \cref{eq:NFWeak} in the sense of distributions. By the initial condition
  in \cref{eq:NFWeak} it follows that 
  \begin{equation} \label{exp2}
    v_{\nu }(x,\xi ,0)=\sum_{k=0}^{\infty }g_{k}(x,0)\psi _{k}(\xi )=v_{0}(x,\xi).
  \end{equation}
On the other hand, applying Parseval's equality and using \cref{reg}, we have
\begin{equation} \label{exp4}
\begin{aligned}
  & v_{0}(x,\xi ) = \sum_{k=0}^{\infty }v_{0,k}(x)\psi _{k}(\xi),
    && v_{0,k}(x)=\bigl(v_{0}(x,\blank),\psi _{k}\bigr)_{L^2(U)}, \\
  & N_{\nu }(x,\xi ,t) =\sum_{k=0}^{\infty }N_{\nu ,k}(x,t)\psi _{k}(\xi),
    && N_{\nu ,k}(x,t)=\bigl(N_{\nu }(x,\blank,t),\psi_{k}\bigr)_{L^2(U)}, 
\end{aligned}%
\end{equation}
which imply
\begin{equation}  \label{exp5}
\begin{aligned}
 & \Vert v_{0}(x)\Vert _{L^{2}(U)}^{2} 
     = \sum_{k=0}^{\infty }v_{0,k}^{2}(x) < \infty, 
  && \text{a.e. $x\in \mathbb{T}^{n}$,} \\
 & \Vert N_{\nu}(x,t)\Vert _{L^{2}(U)}^{2} 
     = \sum_{k=0}^{\infty}N_{\nu ,k}^{2}(x,t) < \infty, 
 && \text{a.e. $(x,t)\in \mathbb{T}^{n}\times \lbrack 0,T],$}
\end{aligned}
\end{equation}
and, by the monotone convergence theorem, we obtain
\begin{equation}\label{exp7}
  \Vert N_{\nu }(x)\Vert_{L^{2}(U_{T})}^{2} 
    =\sum_{k=0}^{\infty}\int_{0}^{T}N_{\nu ,k}^{2}(x,t)dt <\infty, 
    \qquad 
    \text{a.e. $x\in \mathbb{T}^{n}$.} 
\end{equation}%

Integrating \cref{exp5,exp7} over $x\in \mathbb{T}^{n}$ and applying again the
monotone convergence theorem, we have%
\begin{equation} \label{exp10}
 \begin{aligned}
  & \Vert v_{0}\Vert_{L^{2}(\Omega )}^{2} 
    =\sum_{k=0}^{\infty } \hat{v}_{0,k}^{2} < \infty,
   && \hat{v}_{0,k}^{2}=\int_{\mathbb{T}^{n}} v_{0,k}^{2}(x)dx,  \\
  & \Vert N_{\nu} \Vert _{L^{2}(\Omega_{T})}^{2}
    =\sum_{k=0}^{\infty }\int_{0}^{T}\hat N_{\nu ,k}^{2}(t)dt <\infty,
   && \hat N_{\nu ,k}^{2}(t)=\int_{\mathbb{T}^{n}}N_{\nu ,k}^{2}(x,t)dx.
 \end{aligned} 
\end{equation}

Substituting the expansions \crefrange{exp1}{exp4} into system \cref{eq:NFWeak}, we
conclude that for each $k \in \ZSet_{ \geq 0}$, the
coefficients $g_{k}(t)$ satisfy the following initial value problem 
\begin{equation} \label{parabnew1}
  \begin{aligned}
    & \partial_t g_k(x,t) = \lambda_k g_k(x,t) + N_\nu(x,t), 
      && (x,t) \in \TSet^n \times [0,T], \\
    &  g_k(x,0) = v_{0,k}(x), 
      && x \in \TSet^n, 
  \end{aligned}
\end{equation}
and are therefore given by
\begin{equation}\label{parabnew3}
  g_{k}(x,t)=v_{0,k}(x) \exp \bigl(\lambda _{k}t\bigr) 
   + \int_{0}^{t}N_{\nu ,k}(x,s)\exp\bigl( \lambda _{k}(t-s)\bigr) ds,
\end{equation}%
such that $g_{k}(x,\blank) \in H^{1}(0,T)\subset C([0,T])$ for a.e. $x\in
\mathbb{T}^{n}$. This, together with \cref{exp1} completes the expansion of $v_\nu$.
\begin{equation}\label{parabold7}
  v_\nu(x,\xi ,t) = \sum_{k=0}^{\infty} \psi _{k}(\xi )
   \Bigl[v_{0,k}(x)\exp \left( \lambda_k t\right) 
     + \int_{0}^{t}N_{\nu,k}(x,s)\exp \left( \lambda_k(t-s)\right) ds \Bigr].  
\end{equation}
  
\textit{Step 3: expansion of $v$.} The next step is to expand the solution $v$ of the
singular neural field system \cref{nu3} in terms of the eigenfunctions $ \psi
_{k}$. The solution of the problem \cref{nu3} can be written in the form 
\begin{equation} \label{parabnew6}
  v(x,\xi ,t) = v_{0}(x,\xi ) e^{-\gamma t}
    +\int_{0}^{t} N(x,\xi ,s)e^{-\gamma (t-s)} ds.
\end{equation}
By the regularity (\ref{reg}) we have%
\begin{equation} \label{expp}
  N(x,\xi,t)=\sum_{k=0}^{\infty }N_{k}(x,t)\psi_{k}(\xi),
    \qquad N_{k}(x,t)=\bigl(N(x,\blank,t),\psi_{k}\bigr)_{L^2(U)},
\end{equation}
and by Parseval's equality and the monotone convergence theorem
\begin{equation}\label{exp9}
\begin{aligned}
& \Vert N(x,t)\Vert _{L^{2}(U)}^{2} = \sum_{k=0}^{\infty} N_{k}^{2}(x,t) <\infty, 
  && \text{a.e.$(x,t)\in \mathbb{T}^n \times [0,T]$}, \\
& \Vert N(x) \Vert _{L^{2}(U_{T})}^{2} 
  = \sum_{k=0}^{\infty}\int_{0}^{T}N_{k}^{2}(x,t) dt,
  && \text{a.e. $x\in \mathbb{T}^n$},  \\
& \Vert N \Vert _{L^{2}(\Omega_{T})}^{2} 
  = \sum_{k=0}^{\infty}\int_{0}^{T} \hat N_{k}^{2}(t)dt < \infty, 
  && \hat N_{k}^{2}(t)=\int_{\mathbb{T}^{n}}N_{k}^{2}(x,t) dx
\end{aligned}
\end{equation}

Substituting \cref{exp4,expp} into the right-hand side of \cref{parabnew6} we obtain 
\begin{equation}\label{parabnew7}
  v(x,\xi ,t) = \sum_{k=0}^{\infty } \psi _{k}(\xi )
   \biggl[v_{0,k}(x)e^{-\gamma t} +\int_{0}^{t}N_{k}(x,s) e^{-\gamma (t-s)} ds \biggr].  
\end{equation}

\textit{Step 4: expanding $v-v_\nu$ in sums over functions $A_k$ and
$B_k$.} Using the previous
two steps, we will now derive an priori estimate for the difference of $v$ and
$v_{\nu }$ in the norm of the Banach space $L^{2}(U)$.
\changed{
 Henceforth we write}
$-\mu_k(\nu) = -\nu (k\pi/L)^2 = \lambda_k(\nu,\gamma) + \gamma$, where
$\lambda_k$ are the eigenvalues \cref{eq:eienpairs}, and omit the dependence on
$\nu$ when possible.
\changed{
 Combining these definitions} 
with \cref{parabnew7,exp1,parabnew3} we obtain
\[\label{dif1}
  \begin{aligned}
    v(x,\xi ,t)-v_{\nu }(x,\xi ,t) = 
    \sum_{k=0}^{\infty }\psi _{k}(\xi)
    \biggl[
      & v_{0,k}(x) ( 1- e^{-\mu_k t}) e^{-\gamma t} \\
      & + \int_{0}^{t}\bigl( N_{k}(x,s)-N_{\nu ,k}(x,s) e^{-\mu_k (t-s)}\bigr) e^{-\gamma(t-s)} ds
    \biggr] .
  \end{aligned}
\]
Since $\{\psi _{k}\}_{k>0}$ is the orthonormal basis in $L^{2}(U)$, then
taking the $L^{2}(U)$-norm of the previous expressions we obtain 
\begin{equation} \label{dif2}
\Vert v(x,t)-v_{\nu }(x,t)\Vert_{L^{2}(\Omega )}^{2} 
  =\sum_{k=0}^{\infty }(A_{k}+B_{k})^{2}
  \leq 2\sum_{k=0}^{\infty }A_{k}^{2}+2\sum_{k=0}^{\infty }B_{k}^{2},  
\end{equation}
where 
\begin{equation} \label{dif4}
\begin{aligned}
& A_{k}(x,t) = v_{0,k}(x)( 1- e^{-\mu_k t}) e^{-\gamma t}, \\
& B_{k}(x,t) = \int_{0}^{t}\bigl[ N_{k}(x,s)-N_{\nu ,k}(x,s) e^{-\mu_k (t-s)}\bigr] e^{-\gamma(t-s)} ds,
\end{aligned}
\end{equation}
and a further integration over $x \in \TSet^n$ gives
\begin{equation} \label{dif22}
  \Vert v(t)-v_{\nu }(t)\Vert _{L^{2}(\Omega )}^{2}
  \leq 2\sum_{k=0}^{\infty }\int_{\mathbb{T}^{n}} A_{k}^{2}dx
  +2\sum_{k=0}^{\infty }\int_{\mathbb{T}^{n}}B_{k}^{2}dx.
\end{equation}

\textit{Step 5: bounding the sum in $A^2_k$ in the expansion \cref{dif22}}.
By the first equality in \cref{exp10}, we have that for any $\epsilon >0$ there
exists $k_{1}(\epsilon )>0,$ such that
\[
  \sum_{k=k_{1}}^{\infty }\hat{v}_{0,k}^{2}<\frac{\epsilon }{4D},
\]
where the constant $D>0$ will be chosen below, at the very end of the proof. For all
$k \in \ZSet_{\geq 0}$ and $t\in [0,T]$ it
holds $0 \leq ( 1- e^{-\mu_k t}) e^{-\gamma t} \leq 1$, hence we estimate
\begin{equation}\label{dif5}
  \sum_{k=k_{1}}^{\infty }\int_{\mathbb{T}^{n}}A_{k}^{2}(x,t)dx<\frac{\epsilon}{4D}, \qquad t\in [0,T].
\end{equation}
Moreover, for almost any $x \in \TSet^n$, and all $t \in [0,T]$ we have
\[
  \sum_{k=0}^{k_{1}-1}A_{k}^{2}(x,t)
  \leq \sum_{k=0}^{k_{1}-1}|v_{0,k}(x)|^{2} ( 1- e^{-\mu_{k_1} t})^2 e^{-2\gamma t} 
  \leq \Vert v_{0}(x)\Vert _{L^{2}(U)}^{2}( 1- e^{-\mu_{k_1} T})^2,
\]
whence
\begin{equation} \label{dif6}
  \sum_{k=0}^{k_{1}-1}\int_{\mathbb{T}^{n}}A_{k}^{2}(x,t)dx
  \leq \Vert v_{0}\Vert _{L^{2}(\Omega )}^{2}( 1- e^{-\mu_{k_1} T})^2,
  \qquad t \in [0,T]
\end{equation}%

Recalling the dependence of $\mu_k(\nu)$, and using the continuity and monotonicity of
the exponential function, we obtain that for any $\epsi > 0$, there exists
$\nu_1(\epsi)$ such that
\[
  ( 1- e^{-\mu_{k_1}(\nu) T})^2 
  = ( 1- e^{-\nu (k_1\pi/L)^2 T})^2
  < \frac{\epsi}{4 \| v_0 \|^2_{L^2(\Omega)} D}
  \qquad 
  \text{for all $\nu < \nu_1(\epsi)$.}
\]
Using \cref{dif5,dif6}, we deduce that for any $\epsilon >0$  there exists
$\nu_{1}(\epsilon)$ such that
\begin{equation} \label{dif7}
  \sum_{k=0}^{\infty }\int_{\mathbb{T}^{n}}A_{k}^{2}(x,t)dx
  < \frac{\epsilon }{2D}, \qquad 
  \text{for all $\nu <\nu_{1}(\epsilon)$ and $t\in [0,T]$,}
\end{equation}
that is, we have found a bound for the sum in the terms $A_k$ in \cref{dif22}.
  
\textit{Step 6: bounding the sum in $B^2_k$ in the expansion \cref{dif22}.}
\changed{
 With this goal, }
 we rewrite the $B_k$ in \cref{dif4} in the form
\[
 \begin{aligned}
   B_{k}(x,t) 
   & = \int_{0}^{t} N_{k}(x,s)(1- e^{-\mu_k (t-s)}) e^{-\gamma(t-s)} ds \\
   & + \int_{0}^{t} \bigl[ N_{k}(x,s)-N_{\nu ,k}(x,s)\bigr] e^{-\mu_k (t-s)} e^{-\gamma(t-s)} ds,
 \end{aligned} 
\]
Recalling that $(\alpha +\beta )^{2}\leq 2\alpha ^{2}+2\beta ^{2}$, it follows that 
\begin{equation} \label{Bk}
   B^2_{k}(x,t) 
   = 2T(1- e^{-\mu_k T})^2 \int_{0}^{t} N^2_{k}(x,s) ds 
   + 2T \int_{0}^{t} \bigl[ N_{k}(x,s)-N_{\nu ,k}(x,s)\bigr]  ds.
\end{equation}
From the third inequality in \cref{exp9} we deduce that, for any $\epsilon >0$, there
exists an integer $k_2 = k_{2}(\epsilon)$ such that
\[
  \label{dif8}
  \sum_{k=k_{2}}^{\infty}
  \int_{0}^{T} \hat N_{k}^{2}(s)ds
  < \frac{\epsilon }{4D}(2T)^{-1},
\]
hence, integrating \cref{Bk} over $x\in \mathbb{T}^{n}$ we obtain
\[ \label{dif11}
\begin{aligned}
  \sum_{k=0}^{\infty }\int_{\mathbb{T}^{n}}B_{k}^{2}(x,t)dx 
  &
    \leq 2T
    \Biggl(
     \sum_{k=k_{2}}^{\infty } \int_{0}^{T}\hat{N}_{k}^{2}(s)ds
    \Biggr) \\
    & +2T(1- e^{-\mu_{k_2} T})^2 
      \Biggl( \sum_{k=0}^{k_{2}-1}\int_{0}^{T}\hat{N}_{k}^{2}(s)ds\Biggr) \\
    & + 2T \int_{0}^{t} \sum_{k=1}^{\infty} \int_{\TSet^n} \bigl| N_{k}(x,s)-N_{\nu
    ,k}(x,s) \big|^2  dx ds.
  \end{aligned} 
\]
Using the definition of $N$ in \cref{eq:NDef}, and the Lipschitzianity of the function $F$, we estimate 
\[
\begin{aligned}
\sum_{k=0}^{\infty }\int_{\mathbb{T}^{n}}B_{k}^{2}(x,t)dx 
   \leq \frac{\epsi}{4D} 
   & + 2T (1- e^{-\mu_{k_2(\epsi)}(\nu) T})^2 \| G + F(v) \|^2_{L^2(\Omega_T)} \\
    & + 2T K_F \int_{0}^{t} \| u(s)- u_\nu(s) \|^2_{L^2(\Omega)}  ds.
\end{aligned}
\]
We choose $\nu_2(\epsi)$ such that
\begin{equation*}
(1- e^{-\mu_{k_2(\epsi)}(\nu) T})^2 
 \leq 
 \frac{\epsilon }{4D} \left(2T\Vert G + F(v)\Vert _{L^{2}(\Omega _{T})}^{2}\right) ^{-1},
 \qquad \text{for all $\nu < \nu_{2}(\epsilon)$}, 
\end{equation*}
and we conclude
\begin{equation}\label{dif12}
\begin{aligned}
\sum_{k=0}^{\infty }\int_{\mathbb{T}^{n}}B_{k}^{2}(x,t)dx 
   \leq \frac{\epsi}{2D} 
    & + 2T K_F \int_{0}^{t} \| u(s)- u_\nu(s) \|^2_{L^2(\Omega)}  ds.
\end{aligned}
\end{equation}

\textit{Step 7: final estimate.}
We set $\nu (\epsilon )=\min (\nu _{1}(\epsilon ),\nu _{2}(\epsilon ))$, and combine
\cref{dif2,dif7,dif12} to obtain
\[
  \Vert v(t)-v_{\nu }(t)\Vert _{L^{2}(\Omega )}^{2} 
    \leq \frac{\epsilon }{D}
    +2T K_F \int_{0}^{t}\Vert v(s)-v_{\nu }(s)\Vert _{L^{2}(\Omega)}^{2}ds,
    \qquad \text{for all $\nu <\nu (\epsilon )$.}
\]
Applying Gronwall's inequality, recalling that $D$ is a so far unspecified
constant, and picking $D=e^{2T^2K_F}$, we find 
\[
  \Vert v(t)-v_{\nu }(t)\Vert^2_{L^{2}(\Omega )} 
  \leq \frac{e^{2T^2K_F}}{D}\epsilon = \epsi ,\qquad 
  \text{for all $\nu <\nu (\epsilon )$ and $t\in [0,T]$},
\]
hence
\[
  \Vert v-v_{\nu }\Vert _{C([0,T];L^{2}(\Omega ))} \to 0 \qquad \text{as $\nu \to 0$},
\]
and the proof is complete.
\end{proof}

\section{Convergence rate of the sequence $v_\nu - v$ to $0$}\label{sec:convergenceRate} 
The statement below shows that if, in addition to the hypotheses of \cref{the_3
convergence}, one adds regularity assumptions to the initial data and external input,
then it is possible to prove $v_\nu - v$ is $O(\nu^{1/2})$.

\begin{theorem}\label{thm:nuEstimate}
Assume \cref{hyp:general}, and fix $\gamma \geq 0$, $v_0 \in L^2(\Omega)$, and $G \in
L^2(\Omega_T)$. If, in addition, $\partial_\xi v_0 \in L^2(\Omega)$ and $\partial_\xi G \in
L^2(\Omega_T)$, then the
solutions $v$ and $v_{\nu }$ to \cref{nu3,eq:NFWeak}, respectively, satisfy
\begin{equation}
\Vert v_{\nu }-v\Vert _{L^{\infty }(0,T;L^{2}(\Omega ))}
\leq C e^{DT} \nu^{1/2},
\label{rate}
\end{equation}
where positive constants $C$ and $D$ depend on $\gamma$, the constant $K_F$ in
\cref{cor:FEstimates}, and the norms $\Vert v_{0}\Vert _{L^{2}(\Omega)}$, $\Vert
\partial_{\xi} v_{0} \Vert_{L^{2}(\Omega)}$, $\Vert G \Vert _{L^{2}(\Omega_{T})}$,
and $\Vert \partial _{\xi }G\Vert_{L^{2}(\Omega _{T})}$.
\end{theorem}
\begin{proof}
  We initially find an estimate for $v - v_{\nu}$ in a different way to what was
  done in the proof of \cref{the_3 convergence}. This new derivation will reveal that
  estimate \cref{rate} is possible if we gain control over $\partial_\xi v$,
  which we then proceed to bound. 
    
  The difference $w=v_{\nu}-v$ satisfies, for any test function 
  \[
    \begin{aligned}
   &
     \begin{aligned}[b]
       \langle\partial _{t}w,\varphi \rangle
       & + \nu (\partial _{\xi}v_{\nu},\partial_{\xi}\varphi ) \\
       &+\gamma (w,\varphi ) = (F(v_{\nu })-F(v),\varphi ),
       \quad 
  \phi, \partial_\xi \phi \in L^2(\Omega),
     \end{aligned}
   && \text{a.e. in $(0,T)$} \\
   & w =0
   && \text{on $\Omega \times \{ t=0 \}$ }.
    \end{aligned}
  \]
Choosing the test function $\varphi=w$ gives, for almost any in $t \in (0,T)$:
\[
\begin{aligned}
\frac{d}{dt}\Vert w \Vert _{L^{2}(\Omega )}^{2} 
& + \nu \| \partial_\xi v_\nu \|^2_{L^2(\Omega)}
  + \gamma \Vert w \Vert _{L^{2}(\Omega)}^{2} 
  = \bigl(F(v_{\nu})-F(v),w\bigr)
  + \nu (\partial _{\xi }v_{\nu },\partial _{\xi }v) \\
& \leq \frac{2}{\gamma} \Vert F(v_{\nu })-F(v)\Vert _{L^{2}(\Omega )}^{2}
+\frac{\gamma }{2} \Vert w\Vert _{L^{2}(\Omega )}^{2}
+\frac{\nu }{2}\Vert \partial_{\xi }v_{\nu }\Vert_{L^{2}(\Omega )}^{2}
+\frac{\nu }{2}\Vert \partial_{\xi} v\Vert _{L^{2}(\Omega)}^{2}.
\end{aligned}
\]
Disregarding the terms in $\| \partial_\xi v_\nu \|_{L^2(\Omega)}$, integrating
with respect to the time variable over the interval $(0,t)$, and using the
Lipschitzianity of $F$ from \cref{cor:FEstimates}, we obtain
\begin{multline}\label{eq:wEst}
\Vert w(t) \Vert _{L^{2}(\Omega )}^{2} + \int_{0}^{t}\frac{\gamma }{2}
\Vert w(s) \Vert _{L^{2}(\Omega )}^{2}ds \\
 \leq 
C_1 \int_{0}^{t} \Vert w(s)\Vert_{L^{2}(\Omega)}^{2}ds 
+\frac{\nu }{2}  \int_{0}^{t} \Vert \partial _{\xi }v(s)\Vert _{L^{2}(\Omega )}^{2} ds,
\end{multline}
for all $t \in (0,T)$, where $C_1= 2K_F/\gamma$. Inspecting the previous
inequality we note that a bound on $\| \partial_\xi v(t) \|_{L^2(\Omega)}$ is
necessary to find an estimate for $\|w(t) \|^2_{L^2(\Omega)}$. We put the inequality
above \changed{aside, and we devote the rest of the proof to finding} the bound on $\| \partial_\xi v(t)
\|_{L^2(\Omega)}$. 

We begin by extending the singular neural field solution $v$ in the spatial
variable $\xi$. For fixed $\delta >0$, we set 
  \[
    \hat U = (-\delta,L+\delta) \supset U,
    \qquad 
    \hat \Omega = \TSet^n \times \hat U \supset \Omega,
    \qquad 
    \hat \Omega_T = \hat \Omega \times [0,T] \supset \Omega_T.
  \]
  Following the approach of \cite[Theorem 1, p. 268]{Evans},
  we deduce that for any $\delta >0$ there exist functions $\hat{v}_{0}$ and
  $\hat{G}$ such that%
  \[  \label{v0VV}
    \begin{aligned}
    & \hat v_0 \in L^2(\hat \Omega), 
    && \partial_\xi \hat v_0 \in L^2(\hat \Omega), 
    && \hat v_0 = v_0 
    && \text{in $\Omega$}, \\
    & \hat G \in L^2(\hat \Omega_T), 
    && \partial_\xi \hat G \in L^2(\hat \Omega_T), 
    && \hat G = G 
    && \text{in $\Omega_T$}.
    \end{aligned}
  \]
  \changed{
  We can now use Theorem \ref{the_2 copy(1)} to show there exists} a unique weak solution
$\hat{v}$ to the following problem, on the extended domain
\[
  \begin{aligned}
   & \partial_{t}\hat{v} =-\gamma \hat{v}+F(\hat{v})+\hat{G}
   && \text{a.e. in $\hat \Omega_T$}, \\ 
   & \hat{v} = \hat{v}_{0}
   && \text{a.e. in $\hat \Omega \times \{ t =0 \}$}, \\ 
  \end{aligned}
\]
such that 
\begin{equation}\label{up}
  \hat v = v 
  \quad \text{a.e. in $\Omega_T$,}
  \qquad \Vert \hat{v} \Vert_{L^{\infty }(0,T;L^{2}(\Omega ))}^{2}
        +\gamma \Vert \hat{v} \Vert _{L^{2}(\Omega _{T})}^{2} \leq C_2, 
\end{equation}  
for some positive constant $C_2$ which depends only on 
$\Vert \hat{v}_{0}\Vert _{L^{2}(\hat{\Omega})}$, $\Vert \hat{G}\Vert
_{L^{2}(\hat{\Omega }_{T})}$ and $K_F$ defined in \cref{cor:FEstimates}.

We now introduce the following notation: for a function $\hat \psi \colon \hat
\Omega_T \to \RSet$, and a real number $h$ with $|h| < \delta$, we introduce the
function
\[
  \hat \psi_h \colon \Omega_T \to \RSet,
  \qquad 
(x,\xi ,t) \mapsto \hat \psi(x,\xi+h,t)- \hat \psi(x,\xi,t).
\]
The function $\hat{v}_{h}$ satisfies the system
\[
\begin{aligned}
& \partial_{t} \hat{v}_{h} = -\gamma \hat{v}_{h}+
  \bigl[ F(\hat v(\blank,\blank+h,\blank))-F(\hat v)\bigr] +\hat{G}_{h}
   && \text{in $\Omega_T$}, \\ 
 & \hat{v}_{h} =(\hat{v}_{0})_{h}
   && \text{on $\Omega \times \{ t=0 \}$}. \\ 
\end{aligned}%
\]
Multiplying  by $\hat{v}_{h}$, and using Young's inequality for products, we deduce 
\[
\begin{aligned}
  \frac{d}{dt}|\hat{v}_{h}|^{2} +\gamma |\hat{v}_{h}|^{2} 
  & = \bigl([ F(\hat{v}(\blank, \blank +h, \blank))-F(\hat{v})] 
    +\hat{G}_{h}\bigr) \hat{v}_{h} \\
  & \leq 
      \frac{2}{\gamma} \Bigl( 
        | F(\hat{v}(\blank, \blank +h, \blank))-F(\hat{v}) |^{2} 
      + |\hat{G}_{h}|^{2} \Bigr)
      +\frac{\gamma }{2}|\hat{v}_{h}|^{2}
      \qquad \text{a.e. in $\Omega_{T}$}
\end{aligned}
\]
Integrating the last deduced inequality over $\Omega \times (0,t)$ using
\cref{cor:FEstimates} we obtain
\[
  \Vert \hat{v}_{h}(t)\Vert _{L^{2}(\Omega )}^{2}
  \leq C_3\int_{0}^{t} 
  \left( 
    \Vert \hat{v}_{h}(s)\Vert_{L^{2}(\Omega )}^{2}
  + \Vert \hat{G}_{h}(s)\Vert_{L^{2}(\Omega )}^{2}\right) ds 
  + \Vert (\hat{v}_{0})_{h}\Vert _{L^{2}(\Omega)}^{2},
  \quad t\in (0,T),
\]
where $C_3$ depends on $\gamma$ and $K_F$. A further application of
Gronwall's inequality gives
\[
  \Vert \hat{v}_{h} \Vert _{L^{\infty }(0,T;L^{2}(\Omega ))}^{2} 
  \leq e^{C_3 T} C_4 
  \big( \Vert( \hat{v}_{0})_{h} \Vert_{L^{2}(\Omega )}^{2}
    +\Vert \hat{G}_{h}\Vert_{L^{2}(\Omega _{T})}^{2} 
  \big),
\]
where the constant $C_4$ depends on $T$, $\gamma$, and $K_F$. 

Applying \cite[Theorem 4, page 120]{Mih} and the first identity of 
\cref{up}, we conclude that 
\begin{equation*}
\partial _{\xi }v\in L^{\infty }(0,T;L^{2}(\Omega )),
\end{equation*}%
such that
\begin{equation} \label{regV}
\Vert \partial_{\xi } v \Vert _{L^{\infty }(0,T;L^{2}(\Omega ))}^{2}
\leq e^{C_3T}
C_4 \big( 
  \Vert \partial_{\xi} v_{0} \Vert _{L^{2}(\Omega)}^{2}
  +\Vert \partial _{\xi} G \Vert _{L^{2}(\Omega _{T})}^{2}),
\end{equation}
with $C_4$ dependent on $T$, $\gamma$, and $K_F$.

Now that we have bounded $\|\partial_\xi v(t) \|$ uniformly in $t$, we return to
\cref{eq:wEst}. Applying Gronwall's lemma to the inequality obtained from
\cref{eq:wEst} upon disregarding the second term on the left-hand side, and using
\cref{regV}, we obtain
\[
  \begin{aligned}
  \| w(t) \|^2_{L^2(\Omega)} 
  & \leq \nu \frac{e^{C_1 T}}{2} 
    \Vert \partial_{\xi } v \Vert _{L^{\infty }(0,T;L^{2}(\Omega ))}^{2} \\
  & \leq \nu \frac{e^{C_1 T} e^{C_3T}}{2} 
C_4 \big( 
  \Vert \partial_{\xi} v_{0} \Vert _{L^{2}(\Omega)}^{2}
  +\Vert \partial _{\xi} G \Vert _{L^{2}(\Omega _{T})}^{2}),
  \end{aligned}
\]
which proves the estimate \cref{rate}.
\end{proof}
% Plugging the latest estimate in the right-hand side of \cref{eq:wEst} and using 
% \cref{regV} we get
% \[
%   \begin{aligned}
%   \Vert w(t) \Vert _{L^{2}(\Omega )}^{2} 
%   + \frac{\gamma}{2} \Vert w \Vert^2_{L^{2}(\Omega_t)}
%    & \leq 
%      \frac{\nu}{2} \bigl[ 1 + C_1 t e^{C_1 t} \bigr]
%       \Vert \partial_{\xi } v \Vert _{L^{\infty }(0,T;L^{2}(\Omega ))}^{2} \\
%    & \leq 
%      \frac{\nu}{2} 
%  C_4 e^{C_3T} \bigl[ 1 + C_1 T e^{C_1 T} \bigr] 
%  \big[ 
%   \Vert \partial_{\xi} v_{0} \Vert _{L^{2}(\Omega)}^{2}
%   +\Vert \partial _{\xi} G_{h} \Vert _{L^{2}(\Omega _{T})}^{2}],
%   \end{aligned}
% \]
% and the statement follows using the bound $\alpha + \beta  \leq \sqrt{2(\alpha^2 +
% \beta^2)}$ with $\alpha = \|  w(t) \|_{L^2(\Omega)}$ and $\beta = \sqrt{\gamma/2}
% \| w \|_{L^2(\Omega_T)}$.

\section{Numerical results}\label{sec:numerics} 
In this section we present numerical results in support of the theory presented in
the past sections. We simulate system \cref{eq:NFNew} on a rescaled domain 
\[
\TSet = \RSet/2L_{x}\ZSet, \qquad \Omega = \TSet \times (-L_\xi,L_\xi), \qquad
\Omega_T = \Omega \times [0,T],
\]
that is
\begin{equation}\label{eq:NFNewRep}
  \begin{aligned}
   & 
   \begin{aligned}[b]
   \partial_t v(x,\xi,t) 
   = (-\gamma + \nu \partial_{\xi}^{2})v(x,\xi,t) 
   &+ G(x,\xi,t) \\
   &+F(v(\blank,\blank,t))(x,\xi), 
   \end{aligned}
   && (x,\xi,t) \in \Omega_T, 
   \\
   &\partial_\xi v(x,-L_\xi,t) = 0, \qquad  \partial_\xi v(x,L_\xi,t) = 0, 
   &&(x,t) \in \TSet \times [0,T],\\
   &  v(x,\xi,0) = v_0(x,\xi) 
   &&(x,\xi) \in \Omega,
  \end{aligned}
\end{equation}
with the functional and parameter choices similar to \cite{Avitabile.2020}. In
particular, we define
\begin{equation}\label{eq:deltaSpec}
  \delta_\sigma(\xi)= \frac{1}{\sigma\sqrt{\pi}} \exp\biggl(
  -\frac{\xi}{\sigma}^2\biggr), \qquad \sigma \in \RSet_{>0},
\end{equation}
set the nonlinear integral term $F$ as in \cref{eq:FDef} with 
\begin{equation}\label{eq:FSpec}
  \begin{aligned}
  & W(x,\xi,x',\xi') =  \frac{\kappa}{2}\exp(-|x-x'|)\delta_\sigma(\xi-\xi_0)\delta_\sigma(\xi'), 
  &&\xi_0,\kappa \in \RSet \\
  & S(u,\mu,\theta) = \frac{1}{1+\exp(-\mu(u-\theta))},
  &&\mu \in \RSet_{ \geq 0}, \theta \in \RSet, 
  \end{aligned}
\end{equation}
and consider null external input, $G \equiv 0$, and a localised initial condition given by
\begin{equation}\label{eq:v0Spec}
  v_0(x,\xi) =
  \begin{cases}
    \alpha(x) \delta_\sigma(\xi) & \text{if $x >0$,} \\
    \alpha(-x) \delta_\sigma(\xi) & \text{if $x \leq 0$,} \\
  \end{cases}
  \qquad \alpha(x)  = 1 -S(x,\rho,x_0), \qquad \rho,x_0 \in \RSet.
\end{equation}

We simulate the system above using the implicit-explicit time stepper derived in
\cite{Avitabile.2020}, with time step size $\tau = 0.05$, and number of spatial grid
points $n_x =2^{12}$, $n_\xi = 2^{10}$. The initial condition $v_0$ is spatially
localised around the origin. It is close to zero everywhere, except in a small strip,
with characteristic length scale $x_0$ and $\sigma$ in the $x$ and $\xi$ direction,
respectively. Further, we remark that the synaptic kernel is chosen so that
connections are maximal at $\xi = \xi_0$.

We first consider the system in the absence of diffusion, $\nu =0$, as shown in
\cref{fig:profiles}(a). At $t=1$ the activity is localised with two peaks, the first
one centred at $\xi=0$, and the other at $\xi=\xi_0$: the former is induced by the
initial condition $v_0$, while the latter is forming spontaneously, owing to nonlinear
effects and driven by the kernel. The dynamics is dissipative for the first bump,
while the second peak dominates and persists over long time scales. On the other hand
when $\nu$ is small (see \cref{fig:profiles}(b)-(c)) the profiles are merged by diffusion, and the final
profile is wider in the $\xi$ direction, as expected.

\begin{figure}
  \centering
  \includegraphics{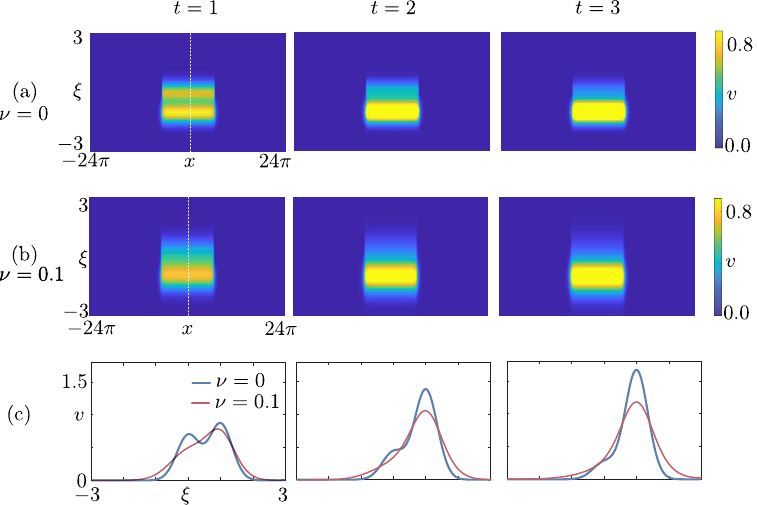}
  \caption{Numerical simulation of model \crefrange{eq:NFNewRep}{eq:v0Spec} in (a) the
    diffusion-less case $\nu=0$, and (b) under weak diffusion $\nu=0.1$, with
    solution profiles for $x=0$ (dashed lines) plotted in (c). In the
    initial stages of the model without diffusion, the voltage displays a profile
    with two bumps, which then evolves toward a single bump, centred at $\xi =\xi_0$.
    As expected, when diffusion is small ($\nu = 0.1$) the two bumps are merged by diffusion, and the
    final profile is wider and with lower amplitude. Other parameters: $\gamma=0.5$,
    $\sigma=0.5$, $\kappa=1$, $\xi_0=1$, $\mu=10^3$, $\theta=0.1$,
$\rho=5$, $x_0=20$, $L_\xi=3$, $L_x=24\pi$, and $T=3$.}
  \label{fig:profiles}
\end{figure}

Further, we can verify the statement of \Cref{thm:nuEstimate}. If we denote by
$v$ and $v_\nu$ the solution to the problem with $\nu=0$ and $\nu \neq 0$,
respectively, the theorem predicts that the square distance 
$e(\nu) = \| v - v_\nu \|^2_{L^\infty(0,T;L^2(\Omega))}$, is an $O(\nu)$ as $\nu \to 0$. In \cref{fig:error}
we provide numerical evidence of this prediction, by repeating the simulation above
for various values of $\{\nu_k\} \subset [0,0.1]$, computing $\{ e(\nu_k) \}$, and
fitting the points $\{ (\nu_k,e(\nu_k)) \}$ with a linear regression.

\begin{figure}
  \centering
  \includegraphics{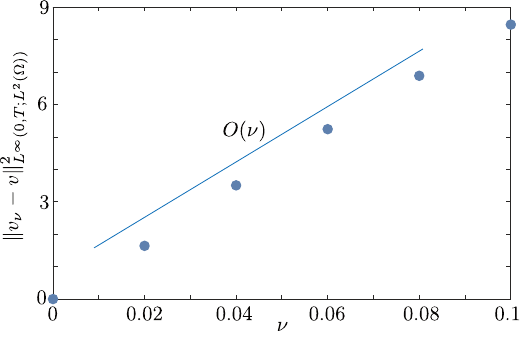}
  \caption{Square distance between solutions with and without diffusion $e(\nu) =
    \| v - v_\nu \|^2$, where the norm is on ${L^\infty(0,T;L^2(\Omega))}$, for various values of
    $\nu \in [0,0.1]$. Points $(\nu_k, e(\nu_k))$ are fit with a linear regression,
    providing numerical evidence that $e(\nu) = O(\nu)$ as $\nu \to 0$, as predicted
    by \Cref{thm:nuEstimate}. Parameters as in \cref{fig:profiles}.}
  \label{fig:error}
\end{figure}

\section{Conclusions}\label{sec:conclusions} 
In this paper we studied the well-posedness and regularity of solutions to neural
field models with dendrites. These models are strongly anisotropic, as diffusion
acts only on one of the spatial direction (the dendritic direction). The particular
laminar structure of the model allows one to characterise solutions as perturbations
to a classical (diffusion-less) neural field models. This structure has already been
exploited numerically in \cite{Avitabile.2020}, but was not supported by analytical
investigations.

Tackling this problem required \changed{studying perturbed} weak solutions to the problem,
and this opens up the possibility of devising schemes in full generality. The
heuristic schemes presented in \cite{Avitabile.2020} relied on collocation on a
simple domain. The weak formulations studied here are the starting point for the
formulation of finite-element schemes, in which the domain $\Omega$ is arbitrary, an
avenue that will be explored in future work.

Furthermore, our study shows that solutions to the dendritic neural fields can be
approximated, in the limit of small diffusion, with classical neural fields, for
which finite-element and spectral methods have recently been developed. We have also
provided numerical evidence that this perturbation results hold on $O(1)$ time
horizons, in spite of the exponentially growing constants predicted by
\cref{thm:nuEstimate}.

This suggests that it should be possible to study splitting schemes in which, at each time
interval, a predictor step is carried out using standard neural field models, while
corrections are made to solve the original problem with diffusion. The latter problem
is relatively cheap to carry out, in view of the anisotropic diffusion operator,
which acts solely along $\xi$. It will be of particular interest to try this
splitting strategy on problems for which diffusion is not necessarily small, and
study the corresponding accuracy of the schemes.
 
\section*{Acknowledgements}
DA acknowledges funding from the Dutch Research Council through the NWO Grant
WC.019.009 ``Spatio-temporal canards and delayed bifurcations in continuous
neurobiological networks". DA's work was also supported by the National Science Foundation
under Grant No. DMS-1929284 while the author was in residence at the Institute for
Computational and Experimental Research in Mathematics in Providence, RI, during the
``Math + Neuroscience: Strengthening the Interplay Between Theory and Mathematics"
programme.
The work of NVC was supported by FAPESP (Funda\c{c}\~{a}o de Amparo \`{a}
Pesquisa do Estado de S\~{a}o Paulo),  project 2021/03758-8, ``Mathematical problems
in fluid dynamics". Also PL and NVC research was supported by FAPESP
(Funda\c{c}\~{a}o de Amparo \`{a} Pesquisa do Estado de S\~{a}o Paulo), through the
Visiting Professor Project  2022/02269-6, ``Numerical and Analytical Investigation of
the Neural Field Equation". The work of PL was supported by FCT within project
UIDB/04621/2020, Center for Computational and Stochastic Mathematics.

\bibliographystyle{siamplain}
\bibliography{references}
\end{document}